\documentclass[11pt,a4paper]{amsart}

\usepackage{amsmath}
\usepackage{amssymb}
\usepackage{amsthm}
\usepackage{enumerate}

\usepackage[all]{xy}
\usepackage{xypic}
\usepackage[mathscr]{eucal}

\usepackage{fullpage}

\theoremstyle{plain} \newtheorem{theorem}{Theorem}
\theoremstyle{plain} \newtheorem{lemma}[theorem]{Lemma}
\theoremstyle{plain} \newtheorem{corollary}[theorem]{Corollary}
\theoremstyle{definition} \newtheorem{remark}[theorem]{Remark}

\newcommand{\mN}{\mathbb{N}}
\newcommand{\mZ}{\mathbb{Z}}
\newcommand{\id}{\operatorname{Id}}
\begin{document}

\title[Reversible posets]{Characterization of hereditarily reversible posets}

\author[Micha{\l} Kukie{\l}a]{Micha{\l} Kukie{\l}a*}

\newcommand{\acr}{\newline\indent}

\address{\llap{*\,}Faculty of Mathematics and Computer Science\acr Nicolaus Copernicus University\acr
 ul. Chopina 12{\slash}18\acr
87-100 Toru\'{n}\acr 
POLAND}

\email{mckuk@mat.umk.pl}

\thanks{The author would like to thank Bernd Schr\"{o}der for reading a preliminary version of this paper. During the preparation of this paper the author has been supported by the joint PhD programme `\'{S}rodowiskowe Studia Doktoranckie z Nauk Matematycznych'.}


\begin{abstract}
A poset $P$ is called reversible if every order preserving bijective self map of $P$ is an order automorphism. $P$ is called hereditarily reversible if every subposet of $P$ is reversible. We give a~complete characterization of hereditarily reversible posets in terms of forbidden subsets. A similar result is stated also for preordered sets. As a corollary we extend the list of known examples of hereditarily reversible topological spaces.
\end{abstract}

\maketitle

\section{Introduction}
It is a common knowledge that, loosely speaking, a bijective morphism need not be an isomorphism, although in many categories (like, say, the categories of sets and functions, groups and homomorphisms, Banach spaces and continuous linear maps) it is so. For example a bijective continuous map between topological spaces may fail to be a homeomorphism. Such phenomena also appear in the category of partially ordered sets (posets) and order preserving maps: consider a bijection from a two element antichain to a two element chain for a simplest example.

Even more is true. One may easily come up with examples of pairs $X,Y$ of non-isomorphic posets (or topological spaces) that are \textit{bijectively related}, i.e. such that bijective maps $X\to Y$ and $Y\to X$ exist.

However, if a poset $X$ is \textit{reversible}, which means that every bijective order preserving map $X\to X$ is an isomorphism, then every poset $Y$ bijectively related to $X$ is in fact isomorphic to $X$. (Reversibility is here a sufficient, but not a necessary condition.) The same may be said about topological spaces and other categories.

The above properties have been studied for example in \cite{doyle,wilansky} in the case of topological spaces. For posets the study was undertaken by the author in \cite{kukiela}. For similar (though not entirely analogous) results in the category of graphs see \cite{bonato} and related papers.

The paper \cite{wilansky} gives a number of examples of hereditarily reversible topological spaces (i.e. spaces containing only reversible subspaces), including only one  example that the authors call non-trivial (which is $\mN\cup\{x\}$ for some $x\in\beta\mN\smallsetminus \mN$). In \cite{kukiela} it was shown that partially well ordered sets are hereditarily reversible. The correspondence between preorders and Alexandroff topological spaces via the specialization preorder (see for example \cite{alexandroff,erne}) allowed the author to give new examples of hereditarily reversible topological spaces as a corollary.

In this note we continue this study, giving a complete characterization of the classes of hereditarily reversible posets and of hereditarily reversible preorders in terms of forbidden subsets. Doing so, we also extend the list of known examples of hereditarily reversible topological spaces.

\section{Main results}
Given a poset $(P,\leq)$ we will omit $\leq$ in notation (i.e. speak of the poset $P$) if this should not lead to confusion.

A poset $P$ is called reversible if every order preserving bijective self map of $P$ is an automorphism. $P$ is called hereditarily reversible if every subposet of $P$ is reversible.

Let $P,Q$ be partially ordered sets and let $p,p'\in P$. By $P\sqcup Q$ we denote the disjoint union and by $P\oplus Q$ the linear sum of $P$ and $Q$. By $P^d$ we denote the dual poset of $P$. If $f\colon P\to P$ is an order preserving bijection and $k\in\mZ$, then by $p^k$ we denote $f^k(p)$. By $p\sim p'$ we denote the fact that $p$ and $p'$ are comparable; $p\!\downarrow$ is the set $\{a\in P:a\leq p\}$; $p\!\uparrow$ is defined dually. For $A\subseteq P$ we define $A\!\downarrow=\cup_{a\in A}a\!\downarrow$ and $A\!\uparrow$ dually. By $p\prec p'$ we denote the fact that $p'$ is a cover of $p$, i.e. a minimal element in $p\!\uparrow$. $\omega$ denotes the least infinite ordinal, $D_\infty$ the antichain of cardinality $\aleph_0$ and $D_1$ the one-element poset.

We define
\begin{align*} F_1&=D_\infty\sqcup\omega, & F_2&=(D_1\oplus D_\infty)\sqcup D_\infty, & F_3&=D_\infty\sqcup \bigsqcup_{i=0}^{\infty}(D_1\oplus D_1),\\
F_4&=D_\infty\oplus\omega, & F_5&=(D_1\sqcup \omega^d)\oplus\omega,
\end{align*}
$F_6,F_7$ are pictured in Figure \ref{fig1} and $F_8$ has the vertex set $\{a_i\}_{i\in\mZ}\cup\{b_i\}_{i\in\mZ}$ ordered by the relation $\leq=\{(a_i,b_j):i>0 \text{ or } j\not=i\}$. 

\begin{figure}[ht]
\centering
$
\xymatrix@R=5mm@C=7mm{
\vdots\ar@{-}[d] \\
\bullet\ar@{-}[d] \\ 
\bullet\ar@{-}[d]\ar@{-}[dr] \\ 
\bullet\ar@{-}[d]\ar@{-}[dr] & \bullet\ar@{-}[d]\ar@{-}[dl] \\ 
\bullet\ar@{-}[d]\ar@{-}[dr] & \bullet\ar@{-}[d]\ar@{-}[dl] \\ 
\vdots & \vdots
}\hspace{3cm}
\xymatrix@R=5mm@C=7mm{
\vdots\ar@{-}[d] \\
\bullet\ar@{-}[d]\ar@{-}[dddrrr] \\ 
\bullet\ar@{-}[d]\ar@{-}[ddrr]  \\ 
\bullet\ar@{-}[d]\ar@{-}[dr] \\
\bullet & \bullet & \bullet & \bullet & \ldots
} 
$
\caption{Posets (from left to right) $F_6$ and $F_7$}\label{fig1}
\end{figure}
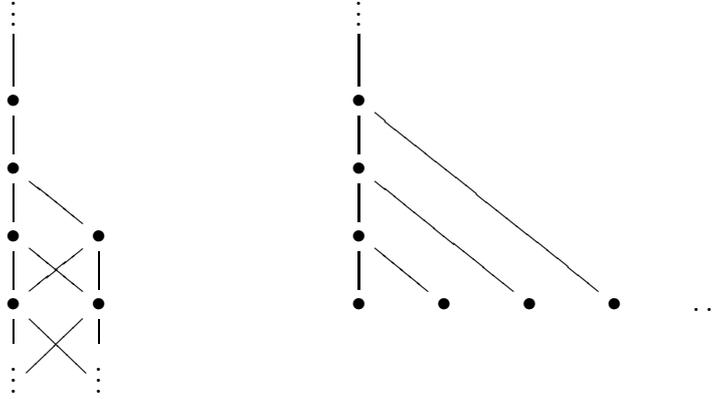

The posets $F_i$ are not reversible and $F_i$ is not isomorphic to a subset of $F_j$ or $F_j^d$ for $i,j=1,\ldots,8$, $i\not=j$. It is interesting that the complements of the comparability graphs of $F_1,F_2,F_3$ are isomorphic to the comparability graphs of, respectively, $F_4$, $F_5$, $F_6$, the comparability graph of $F_7$ is ``nearly'' isomorphic to its own complement (they differ by one vertex) and the complement of the comparability graph of $F_8$ is not a comparability graph. (However, its bipartite complement is isomorphic to the comparability graph of $F_3$.)

\begin{theorem}\label{thm1}
If $P$ is a poset that is not reversible, then $P$ contains a subset isomorphic to $F_i$ or $F_i^d$ for some $i\in\{1,\ldots,8\}$. 
\end{theorem}

Since $F_3$ and $F_8$ are self-dual, there are in total 14 forbidden subsets. Before we prove the theorem, note the following immediate corollary.

\begin{corollary}\label{cor}
Let $P$ be a poset. Then $P$ is hereditarily reversible if no subset of $P$ is isomorphic to any of the sets $F_i$ or $F_i^d$, $i=1,\ldots,8$.
\end{corollary}

The proof of the theorem is based on the following four lemmas, first of which is a well known fact in order theory.

\begin{lemma}\label{main}
Let $X$ be a poset. If $X$ is infinite, then it contains an infinite chain or an infinite antichain.
\end{lemma}

\begin{lemma}\label{infinite_tail}
Let $X$ be a poset and $h\colon X\to X$ an order preserving bijection. Let $x\in X$.
\begin{enumerate}[a.]
%\item \label{pt1} $\{x^{-k}\}_{k\in\mN}$ is infinite iff $\{x^{k}\}_{k\in\mN}$ is infinite iff $x^m=x^n$ for no $m,n\in\mZ$, $m\not=n$.
\item \label{pt2} If $h(x)>x$, then $\{x^{-n}\}_{n\in\mN}$ does not contain a subset isomorphic to $\omega$.
\item \label{pt3} If $x\not\sim x^k$ for any $k\in\mN$, then $\{x^{-n}\}_{n\in\mN}$ is an infinite antichain. Moreover, if $x^k\not\sim x^l$ for any $k,l\in\mN$, $k\not=l$, then $\{x^m\}_{m\in\mZ}$ is an infinite antichain.
\item \label{pt4} If $\{x^{-k}\}_{k\in\mN}$ contains an infinite descending chain, then there is an infinite descending chain in $\{x^{-k}\}_{k\in\mN}\cap x\!\downarrow$.
\end{enumerate}
\end{lemma}
\begin{proof}
\begin{enumerate}[a.]
%\item Clearly, the third condition implies the first two conditions. On the other hand, if $x^m=x^n$, then $x^k=x^{k \operatorname{mod} |m-n|}$, so \[\{x^n\}_{n\in\mZ}\subseteq \{x, x^1,\ldots, x^{|m-n|-1}\}.\]
\item $x^{-n}\not > x^{-k}$ for $n>k$, since otherwise $x> x^{n-k}$.
\item If $x^{-m}\sim x^{-n}$ for some $m,n\in\mN$, then $x\sim x^{|m-n|}$. The second statement follows immediately.
\item Let $\{x^{-k_i}\}_{i\in\mN}$, $k_i<k_{i+1}$ be an infinite descending chain. Then $\{x^{k_0-k_i}\}_{i\in\mN}$ is an infinite descending chain in $x\!\downarrow$. 
\end{enumerate}
\end{proof}

\begin{lemma}\label{no_antichains}
Let $X$ be a poset. If $A\subseteq X$ is isomorphic to $D_\infty$, $C\subseteq X$ is isomorphic to $\omega$ and $c\not\leq a$ for all $a\in A$ and $c\in C$, then $A\cup C$ contains a subset isomorphic to $F_1$, $F_4$ or $F_7$. This happens in particular when, given an order preserving bijection $h\colon X\to X$, $h(x)>x$ for some $x\in X$, $C=\{x^n\}_{n\in\mN}$ and $\{x^{-k}\}_{k\in\mN}$ contains an infinite antichain $A$.
\end{lemma}
\begin{proof}
If there exist infinite subsets $A'\subseteq A$, $C'\subseteq C$ such that $a'\not\sim c'$ for any $a'\in A', c'\in C'$, then $A'\cup C'$ is isomorphic to $F_1$. If there is a $c\in C$ with $c\!\downarrow\cap A$ infinite, then $(c\!\uparrow\cap C)\cup (c\!\downarrow\cap A)$ is isomorphic to $F_4$. Otherwise, it is easy to see $A\cup C$ contains an isomorphic copy of $F_7$. 
\end{proof}

\begin{lemma}\label{skyscraper}
If a poset $X$ may be partitioned into two sets $A=\{a_0>a_1>\ldots\},B=\{b_0>b_1>\ldots\}$that are both isomorphic to $\omega^d$ and $a_i\not\sim b_i$ in $X$ for all $i\in\mathbb{N}$, then there is a subset $C\subseteq X$ such that $C\oplus\omega$ is isomorphic to $F_5$ or $F_6$. 
\end{lemma}
\begin{proof}
Order on $X$ is completely described by the coverings $a_{i+1}\prec a_i$, $b_{i+1}\prec b_i$, $i\in\mathbb{N}$ and coverings of the following two types: $a_{i'}\prec b_{j'}$ for some $i'>j'$ and $b_{i''}\prec a_{j''}$ for some $i''>j''$.

If there are only finitely many coverings of the type $a_{i'}\prec b_{j'}$, then let \[I=\max\{i':\text{there is a covering }a_{i'}\prec b_{j'} \text{ for some } j'\}\] and put $C=\{b_I\}\cup (a_{I}\!\downarrow\cap A)$. If there are finitely many coverings of the type $b_{i''}\prec a_{j''}$, proceed analogously. $C\oplus\omega$ is isomorphic to $F_5$.

If there are infinitely many coverings of both types, then let $k_0=0$. If $k_n$ is defined, let $k_{n+1}=\max(\min\{j:a_{k_n}>b_j\},\min\{j:b_{k_n}>a_j\})$. Put $C=\bigcup_{n\in\mN}\{a_{k_n},b_{k_n}\}$. $C$ is isomorphic to $F_7\oplus\omega$.
\end{proof}

We now proceed to the proof of the theorem.

\begin{proof}[Proof of Theorem \ref{thm1}] Since $P$ is not reversible, an order preserving bijection $f\colon P\to P$ and $u,v\in P$ exist such that $u\not\sim v$ and $f(u)\sim f(v)$. Without loss of generality, we may assume $f(u)>f(v)$. There are four possible relations $<,>,=,\not\sim$ between $f(u)$ and $u$ and between $f(v)$ and $v$, which gives 16 cases. Since $f(u)\leq u$ contradicts $f(v)\geq v$, only 12 cases are left.

By Lemma \ref{no_antichains}, if $f(u)>u$, we may and will assume that $\{u^{-k}\}_{k\in\mN}$ does not contain infinite antichains, and thus, by Lemmas \ref{main} and \ref{infinite_tail}.\ref{pt2},\ref{pt4}, there is an infinite descending chain in $u\!\downarrow$. The same considerations are also true if $u$ is replaced by $v$ and in the dual version. 

If $f(u)\not\sim u$ then we may assume that $u^n\not\sim u$ for all $n\in\mN$, and thus, by Lemma \ref{infinite_tail}.\ref{pt3}, $\{u^{-n}\}_{n\in\mN}$ is an infinite antichain. (Otherwise, replacing $f$ with some $f^n$ reduces the situation to one of the other cases.) Moreover, this implies $u^{-m}\not\sim u^{n}$ for $m,n\in\mN$. If $u^n<u^{n+k}$ or $u^n>u^{n+k}$ for some $n,k\in\mN$, then $\{u^{n+mk}\}_{m\in\mN}$ is an ascending, resp. descending chain, so $\{u^{n+mk}\}_{m\in\mN}\cup\{u^{-m}\}_{m\in\mN}$ is isomorphic to $F_1$, resp. $F_1^d$. Therefore, we will assume $\{u^{m}\}_{m\in\mZ}$ is an infinite antichain. The same holds if we replace $u$ with $v$.

\textit{1st case:} $f(u)>u, f(v)=v$. Since $u\not\sim v$, $u^{-k}\not\sim v$ for $k\in\mN$. Let $\{u^{-k_i}\}_{i\in\mN}$ be an infinite descending chain in $u\!\downarrow$. Then $\{u^{-k_i}\}_{i\in\mN}\cup\{u^n\}_{n\in\mN}\cup\{v\}$ is isomorphic to $F_5$.

\textit{2nd case:} $f(u)>u, f(v)>v$. Let $\{u^{-k_i}\}_{i\in\mN}$ be an infinite descending chain. The set $\{v^{-k_i}\}_{i\in\mN}$ is infinite, so it contains an infinite descending chain $\{v^{-k_{i_j}}\}_{j\in\mN}$. Put $k'_{i_j}=k_{i_j}-k_{i_0}$. Then $\{v^{-k'_{i_j}}\}_{j\in\mN}\subseteq v\!\downarrow$ and $\{u^{-k'_{i_j}}\}_{j\in\mN}\subseteq u\!\downarrow$. Lemma \ref{skyscraper} with $X=\{u^{-k'_{i_j}}\}_{j\in\mN}\cup \{v^{-k'_{i_j}}\}_{j\in\mN}$ says there is a $C\subseteq X$ with $C\cup \{u^n\}_{n\geq 1}=C\oplus \{u^n\}_{n\geq 1}$ isomorphic to $F_5$ or $F_6$.

\textit{3rd case:} $f(u)>u, f(v)<v$. There is an infinite ascending chain $\{v^{-k_i}\}_{i\in\mN}$ in $v\!\uparrow$. If $u>v^n$ for some $n\in\mN$, then $\{u\}\cup \{v^m\}_{m>n}\cup \{v^{-k_i}\}_{i\in\mN}$ is isomorphic to $F_5^d$, since $u\not\sim v^{-k_i}$ for any $i\in\mN$ ($u\geq v^{-k_i}>v$ and $u<v^{-k_i}$, which implies $u<u^{k_i}<v$, are both impossible). If $u>v^n$ for no $n\in\mN$, then $\{v^n\}_{n\geq 1}\cup\{u^n\}_{n\in\mN}$ is isomorphic to $F_5$. 

\textit{4th case:} $f(u)>u, f(v)\not\sim v$. $\{v^{-n}\}_{n\in\mN}$ is an infinite antichain and $v^{-k}>u^n$ is impossible for $k,n\in\mN$, since that would imply $v>u^{n+k}>u$. Thus, by Lemma \ref{no_antichains}, $\{v^{-n}\}_{n\in\mN}\cup \{u^n\}_{n\in\mN}$ contains a subset isomorphic to $F_1$, $F_4$ or $F_7$.

\textit{5th case:} $f(u)<u, f(v)<v$. Analogous to the 2nd case.

\textit{6th case:} $f(u)<u, f(v)\not\sim v$. $\{v^{-n}\}_{n\in\mN}$ is an infinite antichain. Since $v^{-k}\leq u^{n}$ for some $k,n\in\mN$ implies $v\leq u^{n+k}\leq u$, which is impossible, by the dual of Lemma \ref{no_antichains} $\{v^{-k}\}_{k\in\mN}\cup \{u^n\}_{n\in\mN}$ contains a subset isomorphic to $F_1^d$, $F_4^d$ or $F_7^d$.

\textit{7th case:} $f(u)=u, f(v)<v$. Analogous to the 1st case.

\textit{8th case:} $f(u)=u, f(v)\not\sim v$. $\{v^{n}\}_{n\in\mZ}$ is an infinite antichain, so $\{v^{n}\}_{n\in\mZ}\cup\{u\}$ is isomorphic to $F_2^d$.

\textit{9th case:} $f(u)\not\sim u, f(v)=v$. Analogous to the 8th case.

\textit{10th case:} $f(u)\not\sim u, f(v)>v$. Analogous to the 6th case.

\textit{11th case:} $f(u)\not\sim u, f(v)<v$. Analogous to the 4th case.

\textit{12th case:} $f(u)\not\sim u, f(v)\not\sim v$. $\{u^{n}\}_{n\in\mZ}, \{v^{n}\}_{n\in\mZ}$ are both infinite antichains. Note that $v^{n}\geq u^{k}$ for some $n,k\in\mZ$ implies $u^{n+|n|+|k|}\geq v^{n+|n|+|k|}\geq u^{k+|n|+|k|}$, which is, by our assumption, not true. So $\{u^n\}_{n\in\mZ}\cup \{v^n\}_{n\in\mZ}$ is a poset of height 1 with all elements $u^n$, $n\in\mZ$ maximal and all elements $v^n$, $n\in\mZ$ minimal. There are now several cases to consider.

Let $A(u^n)=\{v^m:m\in\mZ, v^m< u^n\}$, $B(u^n)=\{v^m:m\in\mZ, v^m\not< u^n\}$ and $A(v^n)=\{u^m:m\in\mZ, u^m> v^n\}$, $B(v^n)=\{u^m:m\in\mZ, u^m\not> v^n\}$ for $n\in\mZ$. 
\begin{enumerate}[i)]
\item \label{case1} If there is an $n\in\mZ$ such that both sets $A(u^n), B(u^n)$ are infinite, then $\{u^n\}\cup A(u^n)\cup B(u^n)$ is isomorphic to $F_2^d$. The same reasoning applies to the case when $A(v^n)$ and $B(v^n)$ are both infinite for some $n\in\mZ$.
\item \label{case2} If there is no $n$ with $B(u^n)$ infinite, then we may assume there is no $n$ with $B(v^n)$ infinite. Indeed, for fixed $m$ and $n$, $u^m$ is greater than infinitely many $v^k$, $k<n$, so $u^{m+(n-k)} > v^{n}$ for every such $k$. Thus, $A(v^n)$ is infinite. If $B(v^n)$ was infinite, we could reduce to \ref{case1}). Let $n_0=0$. If $n_m$ is defined, then an $n_{m+1}>n_m$ exists such that $u^n, u^{-n}\in A(v^{n_m})\cap A(v^{-n_m}), v^n, v^{-n}\in A(u^{n_m})\cap A(u^{-n_m})$ for all $n\geq n_{m+1}$. The set $\bigcup_{m\in\mN}\{u^{n_m},u^{-n_m},v^{n_m},v^{-n_m}\}$ is isomorphic to $F_8$.
\item \label{case3} If $A(u^N)$ is finite for some $N\in\mZ$, then $A(u^m)$ is finite for all $m\leq N$. Therefore, $A(v^n)\cap \{u^m\}_{m\leq N}$ is finite for every $n\in\mZ$, so $B(v^n)$ is infinite for every $n\in\mZ$. If $A(v^n)$ was infinite for some $n\in\mZ$, we could reduce to \ref{case1}), so we may assume it is finite. Therefore, $B(u^n)$ is infinite for every $n\in\mZ$, so we may assume $A(u^n)$ is finite for every $n\in\mZ$. Let $n_0=0$. If $n_m$ is defined, then an $n_{m+1}>n_m$ exists such that $u^n, u^{-n}\in B(v^{n_m})\cap B(v^{-n_m}), v^n, v^{-n}\in B(u^{n_m})\cap B(u^{-n_m})$ for all $n\geq n_{m+1}$. The set $\bigcup_{m\in\mN}\{u^{n_m},u^{-n_m},v^{n_m},v^{-n_m}\}$ is isomorphic to $F_3$.
\end{enumerate}
\end{proof}

\begin{remark}
Theorem \ref{thm1} and Corollary \ref{cor} can easily be extended to preorders by including the following preordered sets and their duals in the family of forbidden subsets.
\begin{align*} G_1&=Z_\infty\oplus D_\infty, & G_2&=Z_\infty\sqcup D_\infty, & G_3&=Z_\infty\oplus\omega, & G_4&=D_\infty\sqcup\bigsqcup_{i=1}^{\infty}Z_2.
\end{align*}
Here, $Z_\infty$ and $Z_2$ are the preorders of cardinality, respectively, $\aleph_0$ and $2$, such that $p\leq q$ and $q\leq p$ for every pair $p,q$ of their elements.
\end{remark}

\section{Topological consequences}
By specialization preorder of a topological space $(X,\tau)$ we mean the preorder $x\leq_\tau y\iff x\in \overline{\{y\}}$ on the set $X$. The assignment $(X,\tau)\to (X,\leq_\tau)$ is functorial. Recall that given a preorder $(P,\leq)$ one defines the Alexandroff $\mathcal{A}_\leq$ and the upper $\mathcal{U}_\leq$ topologies on the set $P$ to be, respectively, the finest and the coarsest topology on $P$ whose specialization preorder is $\leq$. Concretely, this means that $\mathcal{A}_\leq$ is generated by the basis of open sets $\{x\!\uparrow\}_{x\in P}$, while $\mathcal{U}_\leq$ is generated by the subbasis $\{P\smallsetminus (x\!\downarrow)\}_{x\in P}$. For more details see \cite{erne, kukiela}.

The following may be seen as the extension of \cite[Remark 12]{kukiela}, and thus also of the list of hereditarily reversible topological spaces given in \cite{wilansky}.

\begin{corollary}\label{cor_top}
Let $(P,\tau)$ be a topological space such that $\tau\in\{\mathcal{A}_\leq,\mathcal{U}_\leq\}$ for some preorder $(P,\leq)$. Then $P$ is hereditarily reversible if and only if $(P,\leq)$ is a hereditarily reversible preorder.
\end{corollary}
\begin{proof}
Let $(P,\leq)$ be a reversible preorder. Consider the topological space $(P,\tau)$, with $\tau\in\{\mathcal{A}_\leq,\mathcal{U}_\leq\}$. Let $f\colon P\to P$ be a  bijection which is continuous with respect to $\tau$. Then $f$ preserves the preorder $\leq_\tau=\leq$, and thus $f$ is a preorder isomorphism. But from the definition of $\tau$ it follows that $f$ is also a homeomorphism, and thus $(P,\tau)$ is a reversible topological space. Since for $Q\subseteq P$ with the induced preorder the appropriate Alexandroff or the upper topology agrees with the induced topology from $(P,\tau)$, this proves that $(P,\tau)$ is hereditarily reversible for hereditarily reversible preorders $(P,\leq)$.

On the other hand one easily checks that if $(P,\leq)$ is one of the preorders $F_i$, $i=1,\ldots,8$, or $G_j$, $j=1,\ldots, 4$, or their order duals, then $(P,\tau)$ with $\tau\in\{\mathcal{A}_\leq,\mathcal{U}_\leq\}$ is not a reversible topological space.
\end{proof}

\begin{remark}
A preorder $(P,\leq)$ is reversible if and only if $(P,\tau)$ is a reversible topological space for $\tau=\mathcal{A}_\leq$. However, the `if' part of this statement is not true for $\tau=\mathcal{U}_\leq$ (consider $P=\left(\bigsqcup_{n=1}^{\infty}(\omega\oplus D_\infty)\right)\sqcup \left(\bigsqcup_{n=1}^\infty(\omega\oplus\omega^d)\right)$.) Thus, in the second paragraph of the proof of Corollary \ref{cor_top} we had to take a different route.
\end{remark}

Note that the reasoning in the proof of Corollary \ref{cor_top} can be extended, at least in parts, to other `good enough' topologies $\mathcal{U}_P\subseteq\tau\subseteq\mathcal{A}_P$. In particular, if $\tau$ is such that any order automorphism of $(P,\leq)$ is a homeomorphism of $(P,\tau)$, then reversibility of $(P,\leq)$ implies reversibility of $(P,\tau)$. Moreover, we have the following general fact, which is an extension of \cite[Remark 12]{kukiela} to topologies other than the Alexandroff topology.

\begin{theorem} Let $(X,\tau)$ be a topological space such that $(X,\leq_\tau)$ (or its order dual) is a partially well ordered set (i.e.~a well founded poset without infinite antichains). Then $(X,\tau)$ is hereditarily reversible.
\end{theorem}
\begin{proof}
Let $(X,\leq_\tau)$ be partially well ordered. We only need to prove that $X$ is reversible, since $(A,\leq_\tau)$ is partially well ordered for any subspace $A\subseteq X$.

Let us fix a continuous bijection $f\colon X\to X$. Let $U$ be an open set in $X$. We will show that $f(U)$ is open. Since $U$ is open, it is an up-set in $(X,\leq_\tau)$ (i.e.~$U=U\!\uparrow$). The poset $(X,\leq_\tau)$ is well-founded, so $U=\min(U)\!\uparrow$. But $\min(U)$ is an antichain in $(X,\leq_\tau)$, so it is finite.

Therefore, ordinals $\alpha_1,\ldots,\alpha_n$ exist such that $\min(U)\subseteq X_{\alpha_1}\cup\ldots\cup X_{\alpha_n}=L$, where for an ordinal $\alpha$ the (finite) set $X_{\alpha}=\min\left(X\smallsetminus \bigcup_{\beta<\alpha}X_\beta\right)$ is the $\alpha$-th level set of $(X,\leq_\tau)$ (see for example \cite[p.122]{kukiela}). 

We know that $f$ is an order isomorphism of $(X,\leq_\tau)$ (see \cite[Corollary 11]{kukiela}). In particular, the map $f$ is level-preserving (see for example \cite[proof of Proposition 10]{kukiela}), so $f|_L\colon L\to L$ is a bijection. But $L$ is finite, and thus $(f|_L)^n=\id_L$ for some $n>0$. In particular, $f^n(\min(U))=\min(U)$.

Because $f$ is an order isomorphism, we have \[f^n(U)=f^n(\min(U)\!\uparrow)=f^n(\min(U))\!\uparrow=\min(U)\!\uparrow=U.\] But, since $f$ is continuous, this means that \[f(U)=f^{-(n-1)}(f^n(U))=f^{-(n-1)}(U)\] is an open set.

Replacing `open' by `closed', `up' by `down', `well-founded' by `dually well-founded', `$\min$' by `$\max$' and `$\uparrow$' by `$\downarrow$' we get a proof of reversibility of $(X,\tau)$ for $(X,\leq_\tau)$ dually well ordered.
\end{proof}

%
% ====== Bibliography ==============================================================
%

\end{document}